\documentclass[preprint,review,1p]{elsarticle}
 \usepackage{amssymb}
 \usepackage{amsmath}
 \usepackage{amsthm}
\newtheorem{theorem}{Theorem}[section]

\newtheorem{lemma}[theorem]{Lemma}

\journal{Nonlinear Analysis: Theory, Methods \& Applications}

\begin{document}

\begin{frontmatter}

\title{Thresholds for global existence and blow-up in a general class of doubly dispersive nonlocal wave equations }

\author{H. A. Erbay$^{1}$\corref{cor1}}
    \ead{husnuata.erbay@ozyegin.edu.tr}
\author{S. Erbay$^1$}
    \ead{saadet.erbay@ozyegin.edu.tr}
\author{A. Erkip$^2$}
    \ead{albert@sabanciuniv.edu}
\cortext[cor1]{Corresponding author. Tel: +90 216 564 9489 Fax: +90 216 564 9057}

\address{$^1$ Department of Natural and Mathematical Sciences, Faculty of Engineering, Ozyegin University,  Cekmekoy 34794, Istanbul, Turkey}

\address{$^2$ Faculty of Engineering and Natural Sciences, Sabanci University,  Tuzla 34956,  Istanbul,    Turkey}

\begin{abstract}
    In this article  we study global existence and blow-up of solutions for a general class of nonlocal nonlinear wave equations with power-type nonlinearities, $\displaystyle~ u_{tt}-Lu_{xx}=B(- |u|^{p-1}u)_{xx}, ~~(p>1)$, where the nonlocality enters through two pseudo-differential  operators $L$ and $B$. We establish thresholds for global existence versus blow-up using the potential well method which relies essentially on the ideas suggested by Payne and Sattinger. Our results improve the global existence and blow-up results given in the literature for the present class of nonlocal nonlinear wave equations and cover those given  for many well-known nonlinear dispersive wave equations such as the so-called double-dispersion equation and the traditional Boussinesq-type equations, as special cases.
\end{abstract}

\end{frontmatter}

\setcounter{equation}{0}
\section{Introduction}
\noindent
The present paper  considers the general class of  nonlocal  nonlinear wave equations of the form
\begin{equation}
    u_{tt}-Lu_{xx}=B(g(u))_{xx},   \label{nonlocal} \\
\end{equation}
where $u=u(x,t)$ is a real-valued function, $g(u)=- |u|^{p-1}u$ with $p>1$, and $L$ and $B$ are linear pseudo-differential operators  with smooth symbols $l(\xi )$ and $b(\xi )$, respectively, and identifies sharp thresholds of global existence and blow-up for solutions with subcritical initial energy. The important point to notice here is that this study extends the global existence and blow-up results established recently in  \cite{babaoglu} for (\ref{nonlocal}) to the case of $g(u)=- |u|^{p-1}u~$ $(p>1)$, where  linear semigroups, the contraction mapping principle  and the concavity method of Levine are the main tools for proving the global existence and blow-up results. It is also worth pointing out that the present study uses the potential well method based on the concepts of invariant sets suggested by Payne and Sattinger in \cite{payne} as a result of studying energy level sets.

Throughout this paper it is assumed that $L$ and $B$  are elliptic coercive operators. Denoting the orders of $L$ and $B$  by $\rho$ and $-r$, respectively, with $\rho \geq 0$, $r\geq 0$, this requirement is identified with the existence of  positive constants $~c_{1},c_{2},c_{3}$ and $c_{4}~$ so that
\begin{eqnarray}
 &&   c_{1}^{2}(1+\xi ^{2})^{\rho /2}\leq l(\xi )\leq c_{2}^{2}(1+\xi ^{2})^{\rho/2}, \label{bn-a}\\
 &&  c_{3}^{2}(1+\xi ^{2})^{-r/2} \leq b(\xi )\leq c_{4}^{2}(1+\xi^{2})^{-r/2},  \label{bn-b}
\end{eqnarray}%
for all $\xi \in \mathbb{R}$.

The class of nonlocal nonlinear wave equations characterized by (\ref{nonlocal}) has been introduced recently in \cite{babaoglu} as a generalization of the  so-called double dispersion equation \cite{sam1, sam2}
\begin{equation}
    u_{tt} - u_{xx} - \gamma_{1}u_{xxtt} + \gamma_{2}u_{xxxx} = (g(u))_{xx} \label{double-bouss}
\end{equation}
where $\gamma_{1}>0$ and $\gamma_{2}>0$ are constants, and the terms $u_{xxtt}$ and $ u_{xxxx}$ represent dispersive effects. Notice that setting $~B = (1-\gamma_{1}\partial_x^2)^{-1}~$ and $~L = (1-\gamma_{1}\partial_x^2)^{-1}(1-\gamma_{2}\partial_x^2)~$ in (\ref{nonlocal}) yields (\ref{double-bouss}). For other reduction examples of (\ref{nonlocal}), including the "good", improved or sixth-order Boussinesq equation \cite{bou}, the reader is referred to \cite{babaoglu}. In (\ref{nonlocal}), $B$ is a smoothing operator that smooths out the nonlinear term and it is the source of one type of dispersion. To see the latter fact, we rewrite (\ref{nonlocal}) as $\displaystyle ~B^{-1}u_{tt}-B^{-1}Lu_{xx}=g(u)_{xx}$. Here, the first and second terms on the left-hand side reflect the two sources of dispersive  regularization.

For a general function $g(u)$, the  Cauchy problem of (\ref{nonlocal}) with the initial data
\begin{equation}
    u(x,0)=u_{0}(x),~~u_{t}(x,0)=u_{1}(x),  ~~~~x\in \mathbb{R} \label{ini}
\end{equation}
has been studied in  \cite{babaoglu} and some global existence and blow-up results have been established. To make our  exposition self-contained we repeat the global existence and blow-up results of \cite{babaoglu} without proofs. For that purpose, here are the relevant definitions: $~G(u)=\int_{0}^{u}g(z)dz$,  $~\Lambda^{-1}u={\cal F}^{-1}\left(|\xi|^{-1}{\cal F}u\right)$ (where ${\cal F}$ and ${\cal F}^{-1}$ denote the Fourier transform and its inverse, respectively, in the $x$ variable), and
\begin{equation}
    \mathcal{E}(t)={1\over 2}\left\Vert B^{-1/2}\Lambda^{-1}u_{t}(t) \right\Vert^{2}_{L^{2}}
        +{1\over 2}\left\Vert B^{-1/2}L^{1/2}u(t) \right\Vert^{2}_{L^{2}}+\int_{\mathbb{R}}G(u(t))dx.
        \label{dd-energy}
\end{equation}
Then the two theorems about global existence and blow-up of solutions are as follows:
\begin{theorem}\label{theo1.1}(Theorem 6.4 of \cite{babaoglu})
     Assume that $r+{\rho\over 2} \ge 1$, ${r\over 2}+{\rho\over 2}>{1\over 2}$, $s>{1\over 2}$, $g\in C^{[s]+1}(\mathbb{R})$, $u_{0} \in H^{s}(\mathbb{R})$,    $u_{1}\in H^{s-1-{\rho\over 2}}(\mathbb{R})$, $G(u_{0}) \in L^{1}(\mathbb{R})$ and $G(u) \ge 0$ for all $u\in \mathbb{R}$. Then     the Cauchy problem (\ref{nonlocal}) and (\ref{ini}) has a unique global solution
    $u\in C\left( [0,\infty),H^{s}(\mathbb{R})\right)\cap C^{1}\left( [0,\infty),H^{s-1-{\rho\over 2}}(\mathbb{R})\right)$.
\end{theorem}
\begin{theorem}\label{theo1.2}(Theorem 6.5 of \cite{babaoglu})
    Assume that $B^{-1/2}L^{1/2}u_{0} \in L^{2}(\mathbb{R})$, $B^{-1/2}\Lambda^{-1}u_{1}\in L^{2}(\mathbb{R})$, $G(u_{0}) \in L^{1}(\mathbb{R})$.  If $\mathcal{E}(0) <0$ and there is some $\nu >0$ such that
    \begin{equation}
    ug( u) \leq 2( 1+2\nu) G(u) ~~\mbox{ for all }~~u\in \mathbb{R},  \label{dd-blow-con}
    \end{equation}
    then the solution $u(x,t)$ of the Cauchy problem (\ref{nonlocal}) and (\ref{ini}) blows up in finite time.
\end{theorem}
The above-given theorems are fundamental for describing the behavior of the solutions in many possible cases of the nonlinear function $g(u)$, but they do not cover  the situations addressed in this study for the pure power nonlinearities $g(u)=- |u|^{p-1}u$ with $p>1$. Since $G(u)=-{1\over {p+1}} |u|^{p+1}\le 0$, Theorem \ref{theo1.1} does not cover  this particular form of power nonlinearities. This is one source of the motivation for the present study.
 Another source of the motivation is the restriction $\mathcal{E}(0) <0$ in Theorem \ref{theo1.2}. In the case of $g(u)=- |u|^{p-1}u$, the condition (\ref{dd-blow-con}) of Theorem \ref{theo1.2} holds for $\nu = {{p-1}\over 4}$. However it is unclear how to handle the case $\mathcal{E}(0) >0$ and whether a solution will exhibit finite time blow-up   in such a case. To address  these issues we attempt here to characterize   the dichotomy between global existence  and finite time blow-up in the case of the power nonlinearities.

The aim of this study is twofold: to shed light on the issues raised above for the case of $g(u)=- |u|^{p-1}u$ and to provide sharp thresholds for global existence versus blow-up if the initial energy is strictly below a critical energy constant. The main tool of analysis is the potential well method  based on  the ideas of Payne and Sattinger  \cite{payne} (For a detailed description of the potential well method  the reader is referred, for instance, to \cite{liu1, liu2, yacheng1, yacheng2, ohta, wang}). Noting that the potential energy, namely the last two terms in (\ref{dd-energy}), consists of two parts, the linear part which generates the dispersive effect of the operator $B^{-1}L$ and the purely nonlinear part. It is well-known that the threshold for global existence versus blow-up is determined  by the competition between this type of dispersion and nonlinearity. Thus we construct certain best constants that relate these two effects via  a minimization problem. Then, a critical energy constant $d$ (called the "depth of the potential well" in \cite{payne}) at which the effects due the linear and nonlinear parts of the potential energy are balanced is obtained by solving the minimization problem defined for the total energy functional. Considering the subcritical case, namely, assuming that the total energy is less than $d$, we define two sets of solutions: $\Sigma_{+}$ and   $\Sigma_{-}$. The set $\Sigma_{+}$ corresponds to the case where the linear (dispersive) part dominates the nonlinear part while the set $\Sigma_{-}$ corresponds to the opposite case. We prove that the above two sets of solutions are invariant under the flow generated by (\ref{nonlocal}). Based on this  we establish our global existence and finite time blow-up results. In short, the solution of the Cauchy problem  (\ref{nonlocal}) and (\ref{ini}) exists globally in time if  the initial data lies in $\Sigma_{+}$ and it blows up in finite time if the initial data lies in $\Sigma_{-}$. Finally we extend the analysis to the case of an augmented critical energy constant $d(\gamma)$ resulting from an augmented functional involving a parameter $\gamma$ and we analyze the cases where the parameter dependent results apply whereas those of the parameter independent case fail.

The paper is organized as follows. In Section 2 we cover some preliminaries containing  the local existence theorem of \cite{babaoglu} and the energy and momentum conservation laws.   In Section 3, we first define a constrained minimization problem for the two functionals related the linear (dispersive) and nonlinear parts of the potential energy and  find the critical energy constant.  Then, we define the two invariant sets of solutions  and  establish the threshold for global existence versus blow-up. In Section 4, we extend our considerations to the case of a parameter-dependent objective functional and conclude the section with some closing remarks about the comparison between the threshold of Section 3 and the parameter dependent thresholds of Section 4.

Throughout the paper $\widehat f$ represents the Fourier  transform of $f$, defined by $\widehat f(\xi)=\int_\mathbb{R} f(x) e^{-i\xi x}dx $. For $1\leq p < \infty$ the space $L^p(\mathbb{R})$  denotes the Lebesque space of $p-$integrable functions equipped with the norm $\|f\|_{L^p}$. The inner product of $f$ and $g$ in $L^2(\mathbb{R})$ is indicated by $\langle f, g \rangle$.  Also,  $H^s(\mathbb{R})$ is the Sobolev space for which the norm  $\|f\|_{H^{s}}^2=\int_\mathbb{R} (1+\xi^2)^s |\widehat{f}(\xi)|^2 d\xi$ is finite. $C$ is a generic positive constant.

\setcounter{equation}{0}
\section{Preliminaries}
\noindent
In this section we compile some material on the Cauchy problem  for the doubly dispersive nonlinear  nonlocal equation (\ref{nonlocal}) with $g(u)=-|u|^{p-1} u,~p>1$. For convenience we rewrite (\ref{nonlocal}) as a system of partial differential equations and consider the Cauchy problem
\begin{eqnarray}
    &&  u_t=w_x, ~~~~ x\in \mathbb{R}, ~~~~ t>0,\label{system1} \\
    && w_t= Lu_x+B(g(u))_x, ~~~~ x\in \mathbb{R}, ~~~~ t>0, \label{system2} \\
    && u(x,0)=u_{0}(x),~~w(x,0)=w_{0}(x)~~x\in \mathbb{R}. \label{system3}
\end{eqnarray}
We can now rephrase the local existence theorem of \cite{babaoglu} in terms of the pair $(u, w)$ as follows:
\begin{theorem}\label{theo2.1}
    Assume that $r+{\rho\over 2} \ge 1$,  $s>{1\over 2}$, $g\in C^{[s]+1}(\mathbb{R})$, $(u_{0},w_{0}) \in H^{s}(\mathbb{R})\times H^{s-{\rho\over 2}}(\mathbb{R})$. Then there is some $T_{\max}>0$ such that the Cauchy problem  (\ref{system1})-(\ref{system3}) has a unique solution $(u, w)\in C\left( [0,T_{\max}),H^{s}(\mathbb{R})\right)\times C\left( [0,T_{\max}),H^{s-{\rho\over 2}}(\mathbb{R})\right)$. The maximal time $T_{\max}$ is either $\infty$ or, if finite, is characterized by the blow-up condition
    \begin{displaymath}
        \lim_{t \rightarrow T_{\max}^{-}}~\left[ \left\Vert u(t) \right\Vert_{H^{s}}+\left\Vert w(t)\right\Vert_{H^{s-{\rho\over 2}}}\right]=\infty.
    \end{displaymath}
\end{theorem}
The first statement of Theorem \ref{theo2.1}, namely the local existence result, is indeed rephrasing the above-mentioned result in \cite{babaoglu} in terms of the pair  $(u, w)$. It follows from this local existence result that the solution can be continued beyond $t$ whenever $(u(t), w(t)) \in H^{s}(\mathbb{R})\times H^{s-{\rho\over 2}}(\mathbb{R})$. This in turn gives rise to the second statement about the blow-up condition in finite time.

Recall that Theorem \ref{theo1.1} holds when $r+{\rho\over 2} \ge 1$, ${r\over 2}+{\rho\over 2}>{1\over 2}$, and $s>{1\over 2}$. The important point to note here is that in \cite{babaoglu}, Theorem \ref{theo1.1} was proved in two steps by first considering  the special case $s=s_{0}={r\over 2}+{\rho \over 2}$ and then extending the proof to the general case $s>{1\over 2}$. Thus from now on, without loss of generality, we shall confine our attention to the case $s=s_{0}={r\over 2}+{\rho \over 2}$ and, unifying the second and third conditions above, we shall  assume that $r+{\rho\over 2} \ge 1$ and ${r\over 2}+{\rho\over 2}>{1\over 2}$ (i.e. $s=s_{0}>{1\over 2}$). Here we also note that, since $\rho$ and $r$ are nonnegative, the former inequality implies the latter one except for the case  $(\rho, r)=(0, 1)$. So, in what follows, we make the assumptions
\begin{equation}
    r+{\rho\over 2} \ge 1, ~~~~\rho, r\ge 0, ~~~~(\rho, r)\neq(0,1)  \label{rho-r}
\end{equation}
to simplify the exposition.

We now sketch briefly two conservation laws  which will be important to our analysis.   The laws of conservation of energy and momentum for the system (\ref{system1})-(\ref{system3}) are given by
\begin{eqnarray}
    \mathcal{E}(u(t), w(t))&=&\frac{1}{2} \left\Vert B^{-1/2} w(t)\right\Vert^2_{L^{2}}
                    +\frac{1}{2} \left\Vert B^{-1/2}L^{1/2}u(t)\right\Vert^2_{L^{2}}
                    -\frac{1}{p+1}\left\Vert u(t)\right\Vert^{p+1}_{L^{p+1}}  dx \nonumber \\
                &=&\mathcal{E}(u_{0}, w_{0})  \label{energy}\\
    \mathcal{M}(u(t), w(t))&=&  \int_\mathbb{R}\left( B^{-1/2} w(t)\right) \left(B^{-1/2} u(t)\right) dx=\mathcal{M}(u_{0}, w_{0}), \label{momentum}
\end{eqnarray}
respectively. Note that the space $H^{s_{0}}(\mathbb{R})\times H^{s_{0}-{\rho\over 2}}(\mathbb{R})\equiv H^{{r\over 2}+{\rho\over 2}}(\mathbb{R})\times H^{{r\over 2}}(\mathbb{R})$ is the natural energy space for the pair $(u,w)$. Moreover, since $|\widehat{ (\Lambda^{-1}u_{t})}(\xi)|=|\xi |^{-1}|\widehat{w_{x}}(\xi)|=|\widehat{w}(\xi)|$ where we have used (\ref{system1}), the energy above is the same as  that appearing in (\ref{dd-energy}) for the power nonlinearities. So we refer the reader to Theorem 6.2 of \cite{babaoglu} for a proof of the conservation of energy. To prove the conservation of momentum, we multiply  (\ref{system1}) by $w$ and (\ref{system2}) by $u$ and add the resulting equations. This gives
\begin{displaymath}
    {d\over {dt}}B^{-1}(uw)= B^{-1} (ww_{x})+B^{-1} (uLu_{x})+ug(u)_{x},
\end{displaymath}
from which, by integrating with respect to $x$, we get  ${d\over {dt}}\mathcal{M}(u(t), w(t))=0$.

We close this preliminary section with the following remark. Recall that one of the assumptions of Theorem \ref{theo2.1} is that $g\in C^{[s]+1}(\mathbb{R})$. On the other hand, for the particular case  $g(u)=- |u|^{p-1}u$ considered in this study, we have $g\in C^{\infty}(\mathbb{R})$ when $p$ is an odd integer, $g\in C^{p-1}(\mathbb{R})$ when $p$ is an even integer and $g\in C^{[p]}(\mathbb{R})$   otherwise. Thus, when applied to our particular case, the condition in Theorem \ref{theo2.1} imposes the restriction
\begin{eqnarray*}
  &&  {r\over 2}+{\rho\over 2} \le p-2 ~~~\mbox{if $p$ is an even integer}, \\
  &&  {r\over 2}+{\rho\over 2} \le [p]-1 ~~~\mbox{if $p$ is not an  integer}
\end{eqnarray*}
on $p$.

\setcounter{equation}{0}
\section{Threshold for Global Existence versus Blow-up of Solutions}
\noindent

In this section we show how the potential well method can be used to establish a threshold for global existence versus blow-up of solutions of the Cauchy problem (\ref{system1})-(\ref{system3}). To this end we start by defining two functionals associated with the linear (dispersive) and  nonlinear parts of the potential energy:
\begin{eqnarray}
    &&\mathcal{I}(u )=\frac{1}{2}\int_{\mathbb{R}}(B^{-1/2}L^{1/2}u)^{2}dx,  \label{ic0} \\
    &&\mathcal{Q}(u )=\int_{\mathbb{R}}|u|^{p+1}dx,  \label{momentum-t}
\end{eqnarray}
for  $u\in H^{s_{0}}(\mathbb{R})$. Note that $\mathcal{Q}(u )<\infty$ since  $s_{0}={r\over 2}+{\rho\over 2}>{1\over 2}$ by (\ref{rho-r}).  Now we consider the following constrained variational problem
\begin{equation}
    m= \inf\left\{\mathcal{I}(u ): u\in H^{s_{0}}(\mathbb{R}),~\mathcal{Q}(u)=1\right\}.  \label{var}
\end{equation}
Using the lower and upper estimates given by (\ref{bn-a}) and (\ref{bn-b}) for the symbols $l(\xi)$ and $b(\xi)$ in (\ref{var}) and then using the Sobolev embedding theorem \cite{adams} yields
\begin{eqnarray*}
    \mathcal{I}(u ) \geq {c_{1}^{2}\over {2c_{4}^{2}}} \left\Vert u\right\Vert^2_{H^{s_{0}}}
                    \geq C {c_{1}^{2}\over {2c_{4}^{2}}}\left\Vert u\right\Vert^2_{L^{p+1}}
                    = C {c_{1}^{2}\over {2c_{4}^{2}}} >0
\end{eqnarray*}
from which we deduce that $m>0$. Using the homogeneity of the nonlinear term, the variational problem  (\ref{var}) can also be expressed as follows
\begin{equation}
    \inf \left\{\frac{(\mathcal{I}(u ))^{\frac{p+1}{2}}}{\mathcal{Q}(u)}:  u\in H^{s_{0}}(\mathbb{R}), ~u\neq 0 \right\}=m^{\frac{p+1}{2}}.  \label{variational}
\end{equation}
Note that  $\left(\mathcal{I}(u )\right)^{\frac{1}{2}}$ defines an equivalent norm on $H^{s_{0}}(\mathbb{R})$. In that respect, $m^{-1/2}$  is the best constant for the Sobolev embedding of $H^{s_{0}}(\mathbb{R})$ (endowed with that norm) into $L^{p+1}(\mathbb{R})$.

We now introduce a critical energy constant $d$ to be determined by solving the constrained variational problem:
\begin{equation}
    d=\inf \left\{\mathcal{E}(u,w):~(u,w)\in H^{s_{0}}(\mathbb{R})\times H^{s_{0}-{\rho \over 2}}(\mathbb{R}), ~u \neq 0, ~2\mathcal{I}(u)-\mathcal{Q}(u)=0 \right\}.  \label{critical-energy}
\end{equation}
The constant $d$ so obtained can be viewed as the minimum of the total energy (kinetic plus potential)  for a given potential energy level, and is called the "depth of the potential well" in \cite{payne}. Since $u$ and $w$ are independent in the set over which the infimum is taken, and the constraint is only in $u$, observing that $\mathcal{E}(u, w) \geq \mathcal{E}(u, 0)$ we have
\begin{displaymath}
    d=\inf \left\{\mathcal{V}(u):~u\in H^{s_{0}}(\mathbb{R}), ~u \neq 0, ~2\mathcal{I}(u)-\mathcal{Q}(u)=0 \right\},
\end{displaymath}
where $\mathcal{V}(u)\equiv \mathcal{E}(u, 0)$ denotes the potential energy.  For the static case $u(x, t)=\varphi(x)$, $w(x,t)\equiv 0$, where $\varphi(x)$ represents a stationary solution, the energy is of the form
\begin{equation}
    \mathcal{E}(\varphi,0) = \mathcal{V}(\varphi)= \frac{1}{2}\left\Vert B^{-1/2}L^{1/2}\varphi \right\Vert_{L^2}^2- \frac{1}{p+1}\left\Vert \varphi\right\Vert_{L^{p+1}}^{p+1}
            = \mathcal{I}(\varphi)-\frac{1}{p+1}\mathcal{Q}(\varphi). \label{static-energy}
\end{equation}
Note that for the static case  (\ref{system1})-(\ref{system2}) reduces to  the static equation $B^{-1}L\varphi-|\varphi|^{p-1}\varphi=0$ which is the Euler-Lagrange equation of (\ref{static-energy}).  Integration of this equation in $x$ yields  $2\mathcal{I}(\varphi)-\mathcal{Q}(\varphi)=0$. The important point to note here is that $d=\inf_{\varphi} \mathcal{V}(\varphi)$. Thus, the variational problem (\ref{critical-energy}) can be viewed as an extension  of the one for the static problem.

The following lemma gives the value of $d$ in terms of the best constant for the Sobolev embedding of $H^{s_{0}}(\mathbb{R})$ into $L^{p+1}(\mathbb{R})$.
\begin{lemma}\label{lem3.1}
    $d= \left(\frac{p-1}{p+1}\right) 2^{(\frac{2}{p-1})} m^{(\frac{p+1}{p-1})}$.
\end{lemma}
\begin{proof}
    Let $u\in H^{s_{0}}(\mathbb{R})$ with $2\mathcal{I}(u)=\mathcal{Q}(u)$. By (\ref{variational}) it is obvious that
    \begin{displaymath}
        2m^{(\frac{p+1}{2})}\mathcal{I}(u)=m^{(\frac{p+1}{2})}\mathcal{Q}(u)\leq (\mathcal{I}(u))^{(\frac{p+1}{2})}
    \end{displaymath}
    and so
    \begin{displaymath}
        \mathcal{I}(u)\geq 2^{(\frac{2}{p-1})} m^{(\frac{p+1}{p-1})}.
    \end{displaymath}
    We thus get
    \begin{displaymath}
       \mathcal{V}(u)  = \mathcal{I}(u)-\frac{1}{p+1}\mathcal{Q}(u)
                =\left(1-\frac{2}{p+1}\right)\mathcal{I}(u) \geq \left(\frac{p-1}{p+1}\right) 2^{(\frac{2}{p-1})} m^{(\frac{p+1}{p-1})}.
    \end{displaymath}
    This gives the lower bound for $d$
    \begin{equation}
        d\geq \left(\frac{p-1}{p+1}\right) 2^{(\frac{2}{p-1})} m^{(\frac{p+1}{p-1})}. \label{lower}
    \end{equation}
    Conversely, let $u_{n}$ be a minimizing sequence satisfying $\mathcal{Q}(u_{n})=1$ for the problem (\ref{var}), that is, $\lim_{n\rightarrow \infty} \mathcal{I}(u_{n})=m$. Replacing $u_{n}$ by $\lambda_{n}u_{n}$ we get
    \begin{equation}
        \lim_{n\rightarrow \infty} \mathcal{I}(\lambda_{n}u_{n})= \lambda_n^2 m,~~~~\mathcal{Q}(\lambda_n u_n)=\lambda_n^{p+1}.
    \end{equation}
    Choosing $\lambda_{n}$ so that  $2\mathcal{I}(\lambda_{n}u_{n})-\mathcal{Q}(\lambda_{n}u_{n})=0$ leads to $\lambda_{n}=\left(2\mathcal{I}(u_n)\right)^{\frac{1}{p-1}}$. On the other hand, from (\ref{static-energy}) we have
    \begin{displaymath}
        \mathcal{V}(\lambda_n u_n) = \lambda_{n}^{2}\mathcal{I}(u_n)-\frac{\lambda_n^{p+1}}{p+1}\mathcal{Q}(u_n)
    ~                       = \lambda_{n}^{2}\left(\frac{p-1}{p+1}\right)\mathcal{I}(u_n),
    \end{displaymath}
    and consequently
    \begin{displaymath}
        \lim_{n\rightarrow \infty} \mathcal{V}(\lambda_{n}u_{n})
        =\left(\frac{p-1}{p+1}\right) 2^{(\frac{2}{p-1})} m^{(\frac{p+1}{p-1})}.
    \end{displaymath}
    This gives the upper bound for $d$
    \begin{equation}
        d\leq \left(\frac{p-1}{p+1}\right) 2^{(\frac{2}{p-1})} m^{(\frac{p+1}{p-1})}. \label{upper}
    \end{equation}
    Combining (\ref{lower}) and (\ref{upper}) completes the proof.
\end{proof}
 Together with the critical energy constant $d$, the condition $2\mathcal{I}(u)=\mathcal{Q}(u)$ based on the functionals  $\mathcal{I}(u)$ and $\mathcal{Q}(u)$ related to the dispersive and nonlinear parts of the potential energy determines a balance between the dispersive effect of the operator $B^{-1}L$  and the nonlinear effect. This balance plays a key role in determining the nature of global existence versus blow-up dichotomy for (\ref{nonlocal}).  In that respect we define two sets $\Sigma_{+}$ and $\Sigma_{-}$ as
\begin{eqnarray}
    && \!\!\!\!\!\!\!\!\! \Sigma_{+} =\{ (u,w)\in H^{s_{0}}(\mathbb{R})\times H^{s_{0}-{\rho \over 2}}(\mathbb{R}):\quad \mathcal{E}(u,w) < d, \quad 2 \mathcal{I}(u)-\mathcal{Q}(u) \geq 0 \}, \label{inv1}\\
    && \!\!\!\!\!\!\!\!\! \Sigma_{-} =\{ (u,w)\in H^{s_{0}}(\mathbb{R})\times H^{s_{0}-{\rho \over 2}}(\mathbb{R}):\quad \mathcal{E}(u,w) < d, \quad 2 \mathcal{I}(u)-\mathcal{Q}(u) <0 \}. \label{inv2}
\end{eqnarray}
We note that if $u\neq 0$ and $ 2 \mathcal{I}(u)-\mathcal{Q}(u)=0$ by (\ref{critical-energy}) we have $\mathcal{E}(u,w)\geq d$, hence $(u, w)$ is not in $\Sigma_{+}$. This shows that $\Sigma_{+}$ could be defined alternatively as
 \begin{displaymath}
       \Sigma_{+} =\{ (u,w):\quad \mathcal{E}(u,w) < d, \quad 2 \mathcal{I}(u)-\mathcal{Q}(u) > 0 \}\cup \{ (0,w):\quad \mathcal{E}(0,w) < d \}.
    \end{displaymath}
 Also note that $\Sigma_{+}\cup \Sigma_{-}= \{ (u,w)\in H^{s_{0}}(\mathbb{R})\times H^{s_{0}-{\rho \over 2}}(\mathbb{R}):\quad \mathcal{E}(u,w) < d \}  $.
  Clearly, the crucial fact is that  for both of the sets the energy is subcritical, that is, $\mathcal{E}(u,w)$ is strictly below the critical energy constant $d$. Moreover, due to the second inequality in (\ref{inv2}), $u\neq 0$ for $\Sigma_{-}$. Under this framework, the formula  $(u, w)\in  \Sigma_{+}$ just amounts to saying that   the dispersive effect  dominates over the nonlinear effect. Similar to that,  $(u, w)\in  \Sigma_{-}$ is another way of stating that the nonlinear effect dominates over  the dispersive effect.

The following lemma  shows that the sets $\Sigma_{+}$ and $\Sigma_{-}$ are invariant under the flow generated by (\ref{system1})-(\ref{system2}).
\begin{lemma}\label{lem3.2}
    Suppose $(u_{0},w_{0}) \in \Sigma_{\pm }$, and let $(u(t),w(t))$ be the solution of the Cauchy problem  (\ref{system1})-(\ref{system3}) with initial value $(u_{0},w_{0})$. Then $(u(t),w(t)) \in \Sigma _{\pm }$  for $0<t<T_{\max }$.
\end{lemma}
\begin{proof}
    Let $(u_{0},w_{0}) \in \Sigma_{-}$. Since the energy is conserved, $\mathcal{E}(u(t),w(t))=\mathcal{E}(u_{0},w_{0}) < d$. If  $(u(t),w(t))$  is not in $\Sigma_{-}$, then there is some $t_0>0$ such that for $t\in [0, t_{0})$, $(u(t), w(t))\in \Sigma_{-}$ and   $2\mathcal{I}(u(t_0))=\mathcal{Q}(u(t_0))$. For $t\in [0, t_{0})$, since $2 \mathcal{I}(u)-\mathcal{Q}(u) <0$ and consequently
    \begin{equation*}
        2m^{(\frac{p+1}{2})}\mathcal{I}(u(t)) < m^{(\frac{p+1}{2})}\mathcal{Q}(u(t))
                                            \leq \mathcal{I}(u(t))^{(\frac{p+1}{2})},
    \end{equation*}
    we have
    \begin{equation*}
       \left( \frac{p+1}{p-1}\right)d=2^{(\frac{2}{p-1})}m^{(\frac{p+1}{p-1})}<\mathcal{I}(u(t)).
    \end{equation*}
    By continuity we then have
    \begin{displaymath}
       \mathcal{I}(u(t_{0}))=\lim_{t\rightarrow t_{0}^{-}} \mathcal{I}(u(t)) \geq \left( \frac{p+1}{p-1}\right)d >0,
    \end{displaymath}
    In other words, $u(t_{0})\neq 0$. By (\ref{critical-energy}) we get  $\mathcal{E}(u(t_{0}),w(t_{0})) \geq d$. This is a contradiction, so $(u(t),w(t))\in \Sigma_{-}$. The same argument also works for $\Sigma_{+}$.
    Assume that $(u_0,v_0)\in \Sigma_+$, but $(u(t),w(t))$ does not stay in $\Sigma_+$. Since $\mathcal{E}(u(t),w(t))< d$ there will be some  $t_0$  and $\epsilon>0$ such that $(u(t),w(t))\in \Sigma_+$ for $t\in [0,t_0)$, but $(u(t),w(t))\in \Sigma_-$ for $t\in(t_0,t_0+\epsilon)$.     Then as above,
    $\mathcal{I}(u(t_{0}))=\lim_{t\rightarrow t_{0}^{+}} \mathcal{I}(u(t)) \geq \left( \frac{p+1}{p-1} \right )d >0 $ which implies that $u(t_0) \neq 0$.     Clearly,     $2\mathcal{I}(u(t_0))=\mathcal{Q}(u(t_0))$, which implies the contradiction  $\mathcal{E}(u(t_{0}),w(t_{0})) \geq d$.
\end{proof}

The next two theorems show that the sets $\Sigma_+$ and $\Sigma_-$ determine the life-span of solutions to the Cauchy problem (\ref{system1})-(\ref{system3}) if $\mathcal{E}(u_{0}, w_{0}) < d$. This allows us to split initial data $(u_{0}, w_{0})$ with subcritical energy into two classes. We begin with a global existence result.
\begin{theorem}\label{theo3.3}
     Let $(u_{0},w_{0}) \in \Sigma_{+}$. Then the solution $(u(t),w(t))$ of the Cauchy problem  (\ref{system1})-(\ref{system3}) with initial data $(u_{0},w_{0})$ is global.
\end{theorem}
\begin{proof}
    Let $(u(t),w(t))$ be defined for $t\in \left[0, T_{\max}\right)$. By Lemma \ref{lem3.2}   $(u(t),w(t))\in \Sigma_+$; namely $\mathcal{E}(u(t),w(t)) < d$ and $\mathcal{Q}(u(t)) \leq 2\mathcal{I}(u(t))$ for $t\in \left[0, T_{\max}\right)$. Using (\ref{energy}) in the first inequality yields
    \begin{eqnarray*}
      \mathcal{E}(u(t),w(t))=\frac{1}{2}\left\Vert B^{-1/2}w(t)\right\Vert_{L^2}^2 + \mathcal{I}(u(t)) -\frac{1}{p+1}\mathcal{Q}(u(t))
        < d.
    \end{eqnarray*}
    Combining this with the second inequality $\mathcal{Q}(u(t)) \leq 2\mathcal{I}(u(t))$ gives
    \begin{equation}
       \frac{1}{2}\left\Vert B^{-1/2}w(t)\right\Vert_{L^2}^2+\left(\frac{p-1}{p+1}\right)\mathcal{I}(u(t))  < d. \label{ara1}
    \end{equation}
    By the coercivity of  $L$ and $B$ we have
    \begin{equation}
        {c_{1}^{2}\over {2c_{4}^{2}}} \left\Vert u(t)\right\Vert^2_{H^{s_{0}}}\leq \mathcal{I}(u(t)), ~~~~~
        {1\over {2c_{4}^{2}}} \left\Vert w(t)\right\Vert^2_{H^{s_{0}-{\rho\over 2}}}\leq  \frac{1}{2}\left\Vert B^{-1/2}w(t)\right\Vert_{L^2}^2, \label{ara2}
    \end{equation}
    where we have used (\ref{bn-a}) and (\ref{bn-b}). Combining (\ref{ara1}) and (\ref{ara2}) yields
    \begin{equation*}
        {1\over {2c_{4}^{2}}} \left\Vert w(t)\right\Vert^2_{H^{s_{0}-{\rho\over 2}}}
            +\left(\frac{p-1}{p+1}\right){c_{1}^{2}\over {2c_{4}^{2}}}\Vert u(t)\Vert_{H^{s_{0}}}^2< d.
    \end{equation*}
    This means that $(u(t),w(t))$ stays bounded in $H^{s_{0}}(\mathbb{R})\times H^{s_{0}-{\rho \over 2}}(\mathbb{R})$, thus $T_{\max}=\infty$.
\end{proof}

We now derive a blow-up result using a variation of Levine's Lemma \cite{levine}.
\begin{lemma}\label{lem3.4}
    Suppose that $H(t)$, $t\geq 0$, is a positive, twice differentiable function satisfying  $H''H-(1+\nu )(H')^2\geq 0$ where $\nu >0$. If $H(t_0)>0$ and $H'(t_0)>0$ for some $t_{0}>0$, then $H(t)\rightarrow \infty $  as $t\rightarrow t_1$ for some $t_0<t_1\leq H(t_0)/\left(\nu H'(t_0)\right) $.
\end{lemma}
\begin{theorem}\label{theo3.5}
   Let $(u_{0}, w_{0}) \in \Sigma_{-}$ with $u_0=(v_{0})_{x}$ for some $v_0\in L^2(\mathbb{R})$.  Then the solution $(u(t),w(t))$ of the Cauchy problem (\ref{system1})-(\ref{system3}) with initial data $(u_{0}, w_{0})$ blows up in finite time.
\end{theorem}
\begin{proof}
    Suppose that $u_0=(v_{0})_{x}$ for some  $v_0\in L^{2}(\mathbb{R})$. Then, the formulation of the problem implies that $u=v_x$ where
    \begin{displaymath}
    v(t,.)=v_0+\int_0^t w(\tau,.) d \tau.
    \end{displaymath}
    To obtain blow up in finite time, by  Theorem 2.1 and the remark after the theorem, it suffices to show that $\|w(t)\|_{H^{s_0-\frac{\rho}{2}}}$
    blows up in finite time. By coercivity of $B$, $\|w(t)\|_{H^{s_0-\frac{\rho}{2}}}$ is equivalent to $\|B^{-1/2} w(t)\|_{L^2}$. On the other hand,
    since
    \begin{displaymath}
    \| B^{-1/2} v(t) \|_{L^2} \leq \| B^{-1/2}v_0\|_{L^2}+\int_0^t  \| B^{-1/2}w(\tau)\|_{L^2} d \tau,
    \end{displaymath}
    it will suffice to show that $H(t)=\frac{1}{2}\| B^{-1/2} v(t) \|_{L^2}$ blows up in finite time. Note that $H(t)$ is the same functional
    as in \cite{babaoglu}. We now proceed  as follows:
    \begin{eqnarray*}
        H'(t)   &=& \langle B^{-1/2}v(t), B^{-1/2}v_{t}(t)\rangle , \\
        H^{\prime \prime}(t)
                &=& \left\Vert B^{-1/2}v_t(t)\right\Vert_{L^2}^2+\langle B^{-1/2}v(t), B^{-1/2}v_{tt}(t)\rangle \\
                &=& \left\Vert B^{-1/2}v_t(t)\right\Vert_{L^2}^2 + \int_\mathbb{R} v(t) \left[ B^{-1}L v_{xx}(t) -\left(|v_{x}(t)|^{p-1}v_{x}(t)\right)_{x}\right] dx \\
                &=& \left\Vert B^{-1/2}v_t(t)\right\Vert_{L^2}^2
                    -\left\Vert B^{-1/2}L^{1/2}v_x(t)\right\Vert_{L^2}^{2}
                    +\left\Vert v_x(t)\right\Vert_{L^{p+1}}^{p+1} \\
                &=& \left\Vert B^{-1/2}w(t)\right\Vert_{L^2}^2-2\mathcal{I}(u(t))+\mathcal{Q}(u(t)) .
    \end{eqnarray*}
    We first prove that $H^{\prime \prime }(t) \geq \delta $ for some positive $\delta $. By the conservation of energy, (\ref{energy}),
    \begin{equation*}
        \mathcal{Q}(u(t))
            =\frac{p+1}{2}\left\Vert B^{-1/2}w(t) \right\Vert^{2}_{L^2}
            +\frac{p+1}{2}\left\Vert B^{-1/2}L^{1/2}u(t) \right\Vert^{2}_{L^2}
            -(p+1) \mathcal{E}\left( u_{0}, w_{0}\right).
    \end{equation*}
    Substituting this into $H^{\prime \prime }(t)$ gives
    \begin{eqnarray*}
        H^{\prime \prime}(t)
            &=&\frac{p+3}{2}\left\Vert B^{-1/2}w(t) \right\Vert^{2}_{L^2}
            +(p-1) \mathcal{I}(u(t))-(p+1) \mathcal{E}\left( u_{0}, w_{0}\right).
    \end{eqnarray*}
    Since $2 \mathcal{I}(u)-\mathcal{Q}(u) <0$ and consequently
    \begin{equation*}
        2m^{(\frac{p+1}{2})}\mathcal{I}(u(t)) < m^{(\frac{p+1}{2})}\mathcal{Q}(u(t))
                                            \leq \mathcal{I}(u(t))^{(\frac{p+1}{2})},
    \end{equation*}
    we have
    \begin{equation*}
       \left( \frac{p+1}{p-1}\right)d=2^{(\frac{2}{p-1})}m^{(\frac{p+1}{p-1})}<\mathcal{I}(u(t)).
    \end{equation*}
    Using this result in $H^{\prime \prime }(t)$ yields
    \begin{equation*}
        H^{\prime \prime }(t)\geq \frac{p+3}{2}\left\Vert B^{-1/2}w(t) \right\Vert^{2}_{L^2}+\delta \geq \delta
    \end{equation*}
    where $\delta=(p+1) (d-\mathcal{E}(u_{0}, w_{0}))>0 $.
    Since $\left( H^{\prime }(t) \right)^{2}\leq 2H(t)\left\Vert B^{-1/2}w(t) \right\Vert^{2}_{L^2}$, we have
    \begin{equation*}
    H(t) H^{\prime \prime }(t) -\frac{p+3}{4}(H^{\prime }(t))^{2}
        \geq H(t) \delta  \geq 0.
    \end{equation*}
    Finally, as $H^{\prime \prime }(t) \geq \delta$, there is some $t_{0}\geq 0$ satisfying $H^{\prime }(t_{0}) >0$ and clearly  $H(t_{0}) >0$. By Levine's Lemma, $H(t) $ and hence, by the remark we made at the beginning of the proof,  $(u(t),w(t))$ blows up in finite time.
\end{proof}

\setcounter{equation}{0}
\section{Parameter Dependent Invariant Sets and Thresholds}
\noindent

In this section we improve the results obtained in Section 3 by considering a parameter-dependent functional instead of $\mathcal{I}(u)$ \cite{liu2, yacheng1, yacheng2, ohta, wang}.  We introduce the augmented functional
\begin{displaymath}
    \mathcal{I}_{\gamma}(u)=\frac{1}{2}\int_{\mathbb{R}}\left(B^{-1/2}L^{1/2}u\right)^{2}dx
        -\frac{\gamma^{2}}{2}\int_{\mathbb{R}}\left(B^{-1/2}u\right)^{2}dx,
\end{displaymath}
again for  $u\in H^{s_{0}}(\mathbb{R})$ with $\gamma^{2}< c_{1}^{2}$ where $c_{1}$ is the coercivity constant in (\ref{bn-a}). Clearly,  $\mathcal{I}_{0}=\mathcal{I}$. Note that replacing  $L$  in  $\mathcal{I}$  by the operator $L-\gamma^{2}I$ gives $\mathcal{I}_{\gamma}$ and this creates a new balance between nonlinear and dispersive effects.     An important instance of this type of augmented functional arises when we consider travelling wave solutions of (\ref{nonlocal}).  The analysis is similar in spirit to that of Section 3, we therefore give only the main steps in the proofs.

As in the previous section, we begin by introducing a constrained variational problem
\begin{equation}
    m(\gamma)= \inf\left\{\mathcal{I}_{\gamma}(u ): u\in H^{s_{0}}(\mathbb{R}),~\mathcal{Q}(u)=1\right\}.  \label{var-a}
\end{equation}
 Since
$\mathcal{I}_\gamma(u) \geq {{c_{1}^{2}-\gamma^{2}}\over {2c_{4}^{2}}} \left\Vert u\right\Vert^2_{H^{s_{0}}}$,
as in the previous section, we deduce that $m(\gamma)$ is positive. Once again, using the homogeneity of the nonlinear term, the variational problem (\ref{var-a}) is converted into
\begin{displaymath}
    \inf \left\{\frac{(\mathcal{I}_{\gamma}(u))^{\frac{p+1}{2}}}{\mathcal{Q}(u)}:  u\in H^{s_{0}}(\mathbb{R}), ~u\neq 0 \right\}=\left[m(\gamma)\right]^{\frac{p+1}{2}}.
\end{displaymath}
We now define the augmented critical energy constant $d(\gamma)$
\begin{displaymath}
    d(\gamma)=\inf \left\{\mathcal{E}(u,w)+\gamma \mathcal{M}(u, w):~(u,w)\in H^{s_{0}}(\mathbb{R})\times H^{s_{0}-{\rho \over 2}}(\mathbb{R}),~~~u\neq 0~,~~2\mathcal{I}_{\gamma}(u)-\mathcal{Q}(u)=0 \right\}
\end{displaymath}
and the augmented  sets of solutions
\begin{eqnarray*}
    && \!\!\!\!\!\!\!\!\!\!\!\!
    \Sigma_{+}(\gamma) =\{ (u,w)\in H^{s_{0}}(\mathbb{R})\times H^{s_{0}-{\rho \over 2}}(\mathbb{R}):\quad \mathcal{E}(u,w)+\gamma \mathcal{M}(u, w) < d(\gamma), \quad 2 \mathcal{I}_{\gamma}(u)-\mathcal{Q}(u) \geq 0 \}, \\
    && \!\!\!\!\!\!\!\!\!\!\!\!
    \Sigma_{-}(\gamma) =\{ (u,w)\in H^{s_{0}}(\mathbb{R})\times H^{s_{0}-{\rho \over 2}}(\mathbb{R}):\quad \mathcal{E}(u,w)+\gamma \mathcal{M}(u, w) < d(\gamma), \quad 2 \mathcal{I}_{\gamma}(u)-\mathcal{Q}(u) <0 \}.
\end{eqnarray*}
We note that the key distinction of the critical energy constant $d$ of Section 3 and the augmented critical energy constant $d(\gamma)$ is the condition $\mathcal{E}(u,w)+\gamma \mathcal{M}(u, w) < d(\gamma)$ involving both the energy and the momentum. A further distinction between Sections 3 and 4 is that  the identity
\begin{displaymath}
     \mathcal{E}(u,w)+\gamma \mathcal{M}(u, w)=\frac{1}{2}\left\Vert B^{-1/2}(w+\gamma u) \right\Vert^{2}_{L^2}+\mathcal{I}_{\gamma}(u)-\frac{1}{p+1}\mathcal{Q}(u)
\end{displaymath}
holds here. Once again, the following lemmas play crucial role in establishing global existence and blow-up results. Their proofs proceed along the same lines as those of Lemmas \ref{lem3.1} and \ref{lem3.2}.
\begin{lemma}\label{lem4.1}
    $d(\gamma)= \left(\frac{p-1}{p+1}\right) 2^{(\frac{2}{p-1})} \left[m(\gamma)\right]^{(\frac{p+1}{p-1})}$.
\end{lemma}
\begin{lemma}\label{lem4.2}
    Suppose $(u_{0},w_{0}) \in \Sigma_{\pm }(\gamma)$, and let $(u(t),w(t))$ be the solution of the Cauchy problem  (\ref{system1})-(\ref{system3}) with initial value $(u_{0},w_{0})$. Then $(u(t),w(t)) \in \Sigma_{\pm }(\gamma)$  for $0<t<T_{\max }$.
\end{lemma}
We now state the global existence result, which can be viewed as an augmented version of Theorem \ref{theo3.3}.
\begin{theorem}\label{theo4.3}
     Let $(u_{0},w_{0}) \in \Sigma_{+}(\gamma)$ for some $|\gamma|< c_1$. Then the solution $(u(t),w(t))$ of the Cauchy problem for (\ref{system1})-(\ref{system3}) with initial data $(u_{0},w_{0})$ is global.
\end{theorem}
\begin{proof}
    Let $(u(t),w(t))$ be defined for $t\in \left[0, T_{\max}\right)$. By Lemma \ref{lem4.2}, $(u(t),w(t))\in \Sigma_{+}(\gamma)$. So
    \begin{equation*}
    \mathcal{E}(u(t),w(t))+\gamma \mathcal{M}(u(t), w(t))
       =\frac{1}{2}\left\Vert B^{-1/2}(w(t)+\gamma u(t))\right\Vert_{L^2}^2+\mathcal{I}_{\gamma}(u(t))-\frac{1}{p+1}\mathcal{Q}(u(t))
              < d(\gamma).
    \end{equation*}
    Combining this with $\mathcal{Q}(u(t)) \leq 2\mathcal{I}_{\gamma}(u(t))$ yields
    \begin{equation}
       \frac{1}{2}\left\Vert B^{-1/2}(w(t)+\gamma u(t))\right\Vert_{L^2}^2+\left(\frac{p-1}{p+1}\right)\mathcal{I}_{\gamma}(u(t))  < d(\gamma). \label{ara5}
    \end{equation}
    On the other hand, from the coercivity of $L-\gamma^{2}I$ and $B$ it follows that
    \begin{equation*}
        {{c_{1}^{2}-\gamma^{2}}\over {2c_{4}^{2}}} \left\Vert u(t)\right\Vert^2_{H^{s_{0}}}\leq \mathcal{I}_{\gamma}(u(t)), ~~~~
        {1\over {c_{4}^{2}}} \left\Vert w(t)+\gamma u(t)\right\Vert^2_{H^{s_{0}-{\rho\over 2}}}\leq  \left\Vert B^{-1/2}(w(t)+\gamma u(t))\right\Vert_{L^2}^2.
    \end{equation*}
    Using these two results in (\ref{ara5}) we get
    \begin{equation*}
        {1\over {2c_{4}^{2}}} \left\Vert w(t)+\gamma u(t)\right\Vert^2_{H^{s_{0}-{\rho\over 2}}}
            +\left(\frac{p-1}{p+1}\right){{c_{1}^{2}-\gamma^{2}}\over {2c_{4}^{2}}}\Vert u(t)\Vert_{H^{s_{0}}}^2< d(\gamma),
    \end{equation*}
    from which we conclude that $(u(t),w(t))$ stays bounded in $H^{s_{0}}(\mathbb{R})\times H^{s_{0}-{\rho \over 2}}(\mathbb{R})$ and that $T_{\max}=\infty$.
\end{proof}
The following theorem establishes finite time blow-up of solutions and, once again, relies on Lemma \ref{lem3.4}:
\begin{theorem}\label{theo4.4}
   Let $(u_{0}, w_{0}) \in \Sigma_{-}(\gamma)$ for some $|\gamma|< c_1$, with $u_0=(v_{0})_{x}$ for some $v_0\in L^2(\mathbb{R})$  and  $\gamma \mathcal{M}(u_{0}, w_{0})\geq 0$.  Then the solution $(u(t),w(t))$ of the Cauchy problem (\ref{system1})-(\ref{system3}) with initial data $(u_{0}, w_{0})$ blows up in finite time.
\end{theorem}
\begin{proof}
    Since the proof is similar to that of Theorem \ref{theo3.5}, we only point out some modifications. As before we consider the functional
    $H(t) =\frac{1}{2}\left\Vert B^{-1/2}v(t)\right\Vert_{L^2}^2$. Then
    \begin{eqnarray*}
        H'(t)   &=& \langle B^{-1/2}v(t), B^{-1/2}v_{t}(t)\rangle, \\
        H^{\prime \prime}(t)
                &=& \left\Vert B^{-1/2}v_{t}(t)\right\Vert_{L^2}^2
                    -\left\Vert B^{-1/2}L^{1/2}v_{x}(t)\right\Vert_{L^2}^{2}+\left\Vert v_{x}(t)\right\Vert_{L^{p+1}}^{p+1} \\
               &=&\frac{p+3}{2}\left\Vert B^{-1/2}w(t) \right\Vert^{2}_{L^2}
                     +\frac{(p-1)}{2}\gamma^2  \left\Vert B^{-1/2}u(t)\right\Vert_{L^2}^{2}
                    +(p-1) \mathcal{I}_{\gamma}(u(t))-(p+1) \mathcal{E}\left( u_{0}, w_{0}\right ).
    \end{eqnarray*}
    Since
    \begin{equation*}
        \left(\frac{p+1}{p-1}\right)d(\gamma)=2^{(\frac{2}{p-1})}\left[m(\gamma)\right]^{(\frac{p+1}{p-1})}<\mathcal{I}_{\gamma}(u(t)),
    \end{equation*}
    we have
    \begin{eqnarray*}
        H^{\prime \prime }(t)&=& \frac{p+3}{2}\left\Vert B^{-1/2}w(t) \right\Vert^{2}_{L^2}
            +\frac{p-1}{2}\gamma^2  \left\Vert B^{-1/2}u(t)\right\Vert_{L^2}^{2} \\
           &~& +(p+1)\gamma \mathcal{M}(u_{0}, w_{0})+(p+1)\left[d(\gamma)-\mathcal{E}\left( u_{0}, w_{0}\right)-\gamma \mathcal{M}(u_{0}, w_{0})
            \right].
    \end{eqnarray*}
    Finally, as $\gamma \mathcal{M}(u_{0}, w_{0})\geq 0$, we have
    \begin{equation*}
    H^{\prime \prime }(t) \geq \frac{p+3}{2}\left\Vert B^{-1/2}w(t) \right\Vert^{2}_{L^2}+\delta \geq \delta,
    \end{equation*}
    with $\delta=\left[d(\gamma)-\mathcal{E}\left( u_{0}, w_{0}\right)-\gamma \mathcal{M}(u_{0}, w_{0})\right]>0 $.  The rest of the proof is the same as in Theorem \ref{theo3.5}.
\end{proof}

Finally, we compare the present results with those of Section 3 and make some closing remarks:
\begin{itemize}
\item A simple calculation of the Euler-Lagrange equations shows that  minimizers for $\mathcal{I}(u)$ and $\mathcal{I}_{\gamma}(u)$ are in fact standing and traveling wave solutions of (\ref{nonlocal}), respectively. The Boussinesq equation and its certain generalizations have explicitly known traveling wave solutions and this enables one to compute $d(\gamma)$ explicitly. However, in the general case that we have considered in this paper, finding explicit solutions is obviously not possible.
\item We emphasize that Theorem \ref{theo4.3} extends Theorem \ref{theo3.3} even though their statements are very similar. We now consider the sets $\Sigma_{+}$ and $\Sigma_{+}(\gamma)$  to briefly explain the argument related to  Theorems \ref{theo3.3} and \ref{theo4.3}. Clearly for $u \neq 0$ and $\gamma \neq 0$ we have $\mathcal{I}_{\gamma}(u) < \mathcal{I}(u)$ and thus $d(\gamma) < d (0) = d$. So the constraint $2\mathcal{I}_{\gamma}(u)\geq  \mathcal{Q}(u)$ of the set $\Sigma_{+}(\gamma)$ implies that of $\Sigma_{+}$, namely $2\mathcal{I}(u)\geq \mathcal{Q}(u)$. On the other hand, we note that the inequality $\mathcal{E}(u, w)+\gamma \mathcal{M}(u,w) < d(\gamma)$ of the set $\Sigma_{+}(\gamma)$, that is, $\mathcal{E}(u, w) < d(\gamma)-\gamma \mathcal{M}(u,w)$ implies $\mathcal{E}(u, w) < d$ if  $ d(\gamma)-\gamma \mathcal{M}(u,w)\leq d$. So, for $(u_{0},w_{0})\in \Sigma_{+}(\gamma)$  we get $(u_{0},w_{0})\in \Sigma_{+}$ if $d(\gamma)-\gamma \mathcal{M}(u_{0},w_{0})\leq d$, hence Theorem \ref{theo3.3} already applies. We conclude that for the initial data with $d(\gamma)-\gamma \mathcal{M}(u_{0},w_{0})> d$ the conclusion of Theorem \ref{theo4.3} does not follow from Theorem \ref{theo3.3}. For this case Theorem \ref{theo4.3} is an improved result.
\item We make a similar observation about Theorems \ref{theo3.5} and \ref{theo4.4}. Since $\gamma \mathcal{M}(u_{0},w_{0})\geq 0$ we have
    \begin{displaymath}
    \mathcal{E}(u_{0},w_{0})< d(\gamma) - \gamma \mathcal{M}(u_{0},w_{0})<d-\gamma \mathcal{M}(u_{0},w_{0}) .
    \end{displaymath}
    Thus the  level of the energy $\mathcal{E}(u_{0},w_{0})$ for the solutions in the set $\Sigma_{-}(\gamma)$  is lower than that of $\Sigma_{-}$. On the other hand, as $\mathcal{I}_{\gamma}(u) < \mathcal{I}(u)$ for $u \neq 0$ and $\gamma \neq 0$  Theorem \ref{theo4.4} sets a new balance between nonlinear and dispersive effects and one may have blow-up even for smaller $\mathcal{Q}(u_{0})$.
\item Finally we point out that, in comparison with Theorem \ref{theo3.5}, we prove Theorem \ref{theo4.4} under the extra assumption $\gamma \mathcal{M}(u_{0},w_{0})\geq 0$. Ideally one would expect a blow-up result in Theorem \ref{theo4.4} as in Theorem \ref{theo3.5}; but this cannot be generally true and stability results for traveling waves provide a clue in that direction. As in \cite{liu2} a blow up result like that of Theorem \ref{theo3.5} can be used to show the orbital instability of traveling waves. On the other hand, for the Boussinesq and similar equations it is known that traveling waves are orbitally stable whenever $d(\gamma)$ is convex which in turn holds for sufficiently away from $0$. Hence one cannot expect blow-up without some extra assumption as $\gamma \mathcal{M}(u_{0},w_{0})\geq 0$ for $\gamma$ within that range.
\end{itemize}
\vspace*{10pt}

\noindent
{\bf Acknowledgement}: This work has been supported by the Scientific and Technological Research Council of Turkey (TUBITAK) under the project TBAG-110R002.


\end{document}